\documentclass[12pt]{article}
\usepackage{amsmath}
\usepackage{amssymb}
\usepackage{amsthm}
\providecommand{\abs}[1]{\lvert#1\rvert}

\newtheorem{lemma}{Lemma}
\newtheorem{theorem}{Theorem}

\newtheorem{corollary}{Corollary}
\newtheorem{question}{Question}

\newcommand{\ov}{\overline}
\newcommand{\cb}{{\cal B}}
\newcommand{\ce}{{\cal E}}
\newcommand{\cf}{{\cal F}}
\newcommand{\ch}{{\cal H}}
\newcommand{\cl}{{\cal L}}
\newcommand{\ra}{\rightarrow}
\newcommand{\Ra}{\Rightarrow}
\newcommand{\dist}{\mbox{\em dist\/}}

\makeatletter
\newcounter{subclaim}
\newenvironment{subclaim}{\stepcounter{subclaim}\begin{itemize}\item[(\roman{subclaim})]\def\@currentlabel{(\roman{subclaim})}}{\end{itemize}}
\let\@oldproof=\proof\def\proof{\setcounter{subclaim}{0}\@oldproof} 
\makeatother

\begin{document}

\begin{center}
{\bf\Large Lines in hypergraphs}\\
\vspace{0.5cm}
Laurent Beaudou (Universit\' e Blaise Pascal, Clermont-Ferrand)\footnote{{\tt laurent.beaudou@ens-lyon.org}}\\
 Adrian Bondy (Universit\' e Paris 6)\footnote{{\tt adrian.bondy@sfr.fr}}\\
 Xiaomin Chen (Shanghai Jianshi LTD)\footnote{{\tt gougle@gmail.com}}\\
 Ehsan Chiniforooshan (Google, Waterloo)\footnote{{\tt chiniforooshan@alumni.uwaterloo.ca}}\\
 Maria Chudnovsky (Columbia University, New York)\footnote{{\tt mchudnov@columbia.edu}\\
\hphantom{lllllxxx}Partially supported by NSF grants DMS-1001091 and IIS-1117631}\\
 Va\v sek Chv\' atal (Concordia University, Montreal)\footnote{{\tt chvatal@cse.concordia.ca}\\
\hphantom{lllllxxx}Canada Research Chair in Combinatorial Optimization}\\
 Nicolas Fraiman (McGill University, Montreal)\footnote{{\tt nfraiman@gmail.com}}\\
 Yori Zwols (Concordia University, Montreal)\footnote{{\tt yzwols@gmail.com}}
\end{center}

\begin{center}
{\bf Abstract}
\end{center}
{\small One of the De Bruijn - Erd\H os theorems deals with finite
  hypergraphs where every two vertices belong to precisely one
  hyperedge. It asserts that, except in the perverse case where a
  single hyperedge equals the whole vertex set, the number of
  hyperedges is at least the number of vertices and the two numbers
  are equal if and only if the hypergraph belongs to one of simply
  described families, near-pencils and finite projective planes.  Chen
  and Chv\' atal proposed to define the line $uv$ in a $3$-uniform
  hypergraph as the set of vertices that consists of $u$, $v$, and all
  $w$ such that $\{u,v,w\}$ is a hyperedge. With this definition, the
  De Bruijn - Erd\H os theorem is easily seen to be equivalent to the
  following statement: If no four vertices in a $3$-uniform hypergraph
  carry two or three hyperedges, then, except in the perverse case
  where one of the lines equals the whole vertex set, the number of
  lines is at least the number of vertices and the two numbers are
  equal if and only if the hypergraph belongs to one of two simply
  described families. Our main result generalizes this statement by
  allowing any four vertices to carry three hyperedges (but keeping
  two forbidden): the conclusion remains the same except that a third
  simply described family, complements of Steiner triple systems,
  appears in the extremal case.}

\section{Introduction}\label{sec.intro}

Two distinct theorems are referred to as ``the De Bruijn - Erd\H os
theorem''. One of them \cite{DE51} concerns the chromatic number of
infinite graphs; the other \cite{DE48} is our starting point:
\begin{quote}{\em Let $m$ and $n$ be positive integers such that $n\ge
    2$; let $V$ be a set of $n$ points; let $\cl$ be a family of $m$
    subsets of $V$ such that each member of $\cl$ contains at least
    two and at most $n-1$ points of $V$ and such that every two points
    of $V$ belong to precisely one member of $\cl$. Then $m\ge n$,
    with equality if and only if 

    one member of $\cl$ contains $n-1$ points of $V$ and each of the
    remaining $n-1$ members of $\cl$ contains two points of $V$ 

    or else $n=k(k-1)+1$, each member of $\cl$ contains $k$ points of
    $V$, and each point of $V$ is contained in $k$ members of $\cl$.  }
\end{quote}
We study variations on this theme that are generated through the
notion of lines in hypergraphs. A {\em hypergraph\/} (the term comes
from Claude Berge~\cite{Ber}) is an ordered pair $(V,\ce)$ such that
$V$ is a set and $\ce$ is a set of subsets of $V$; elements of $V$ are
the {\em vertices\/} of the hypergraph and elements of $\ce$ are its
{\em hyperedges;\/} a hypergraph is called {\em $k$-uniform\/} if all
its hyperedges have precisely $k$ vertices.  Given a $3$-uniform
hypergraph and its distinct vertices $u,v$, Chen and Chv\'
atal~\cite{CC} define the {\em line\/} $\ov{uv}$ as the set of
vertices that consists of $u$, $v$, and all $w$ such that $\{u,v,w\}$
is a hyperedge. (When $V$ is a subset of the Euclidean plane and $\ce$
consists of all collinear triples of vertices, $\ov{uv}$ is the
intersection of $V$ and the Euclidean line passing through $u$ and
$v$.)

\bigskip

If, as in the hypothesis of the De Bruijn - Erd\H os theorem,
$(V,\cl)$ is a hypergraph in which each hyperedge contains at least
two vertices and every two vertices belong to precisely one
hyperedge, then $\cl$ is the set of lines of a $3$-uniform hypergraph
$(V,\ce)$: to see this, let $\ce$ consist of all the three-point
subsets of all hyperedges in $\cl$. As for the converse of this
observation, if $(V,\ce)$ is a $3$-uniform hypergraph, then each of
its lines contains at least two vertices and every two 
vertices belong to at least one line, but they may belong to more than
one line.  For example, if $V$ contains distinct vertices 
 $p,q,r,s$ such
that $\{p,q,r\}\in\ce$, $\{p,q,s\}\in\ce$, $\{p,r,s\}\not\in\ce$, then
lines $\ov{pr}$, $\ov{ps}$ are distinct and $p,q$ belong to both of them. 
Now we are going to show that this is the only example.

\begin{theorem}\label{thm.no2no3}
  If, in a {\rm $3$}-uniform hypergraph $(V,\ce)$, some two vertices
  belong to more than one line, then $V$ contains distinct vertices
  $p,q,r,s$ such that $\{p,q,r\}\in\ce$, $\{p,q,s\}\in\ce$,
  $\{p,r,s\}\not\in\ce$.
\end{theorem}

\begin{proof}
  Let $(V,\ce)$ be a $3$-uniform hypergraph. Assuming that two of its
  vertices, $u$ and $v$, belong not only to the line $\ov{uv}$, but
  also to some other line $\ov{xy}$, we will find distinct vertices
  $p,q,r,s$ such that at least two but not all four of
  $\{p,q,r\}$, $\{p,q,s\}$, $\{p,r,s\}$, $\{q,r,s\}$ belong to $\ce$.

\medskip

{\sc Case 1:} {\em One of $x,y$ is one of $u,v$.\/}\\
Symmetry lets us assume that $x=u$; now $\ov{uy}\ne \ov{uv}$ and
$v\in\ov{uy}$, and so $y\ne v$ and $\{u,v,y\} \in \ce$. Since
$\ov{uy}\ne \ov{uv}$, some vertex $z$ belongs to precisely one of
these two lines; since $\{u,v,y\} \in \ce$, the vertices $u,v,y,z$ are
all distinct. Since precisely one of $\{u,y,z\}$ and $\{u,v,z\}$
belongs to $\ce$, we may take $u,v,y,z$ for $p,q,r,s$.

\medskip

{\sc Case 2:} {\em $x,y,u,v$ are all distinct.\/}\\
Since $u\in\ov{xy}$ and $v\in\ov{xy}$, we have $\{u,x,y\} \in \ce$ and
$\{v,x,y\} \in \ce$, If $\{u,v,x\} \not\in \ce$ or $\{u,v,y\} \not\in
\ce$, then we may take $u,v,x,y$ for $p,q,r,s$; if $\{u,v,x\} \in \ce$
and $\{u,v,y\} \in\ce $, then we are back in Case 1 with $(v,x)$ in
place of $(x,y)$ if $\ov{uv}\ne\ov{vx}$ and with $(v,x)$ in place of
$(u,v)$ if $\ov{uv}=\ov{vx}$.
\end{proof}

\medskip

When $W\subseteq V$, the {\em sub-hypergraph of $(V,\ce)$ induced by
  $W$\/} is $(W,\cf)$ with $\cf$ consisting of all elements of $\ce$
that are subsets of $W$. In this terminology, Theorem~\ref{thm.no2no3} states that
\begin{equation}\label{no2no3--dbe}
\begin{array}{l}
\mbox{\em in a {\rm $3$}-uniform hypergraph, no sub-hypergraph induced}\\
\mbox{\em by four  vertices has two or three hyperedges if and only if}\\
\mbox{\em every two vertices belong to precisely one line.}
\end{array}
\end{equation}
We will say that a $3$-uniform hypergraph has the {\em De Bruijn -
  Erd\H os property\/} if it has at least as many distinct lines as it
has vertices or else one of its lines consists of all its vertices.
In this terminology, an immediate corollary of (\ref{no2no3--dbe}) and
the De Bruijn - Erd\H os theorem states that
\begin{equation}\label{no2no3dbe}
\begin{array}{l}
\mbox{\em if, in a {\rm $3$}-uniform hypergraph, no sub-hypergraph induced}\\
\mbox{\em by four  vertices has two or three hyperedges,}\\
\mbox{\em then the hypergraph has the De Bruijn - Erd\H os property.}
\end{array}
\end{equation}

Not every $3$-uniform hypergraph has the De Bruijn - Erd\H os
property: here is a $3$-uniform hypergraph $(V, \ce)$ with
$\abs{V}=11$ that has precisely ten distinct lines and none of these
lines equals $V$. Its vertex set $V$ is
\[
\{1,2\}\cup (\{a,b,c\}\times\{d,e,f\});
\]
its hyperedges are the $\binom{9}{3}$ three-point subsets of
$\{a,b,c\}\times\{d,e,f\}$ and the $18$ three-point sets
$\{i,(x_1,x_2),(y_1,y_2)\}$ with $x_i=y_i$; its ten lines are
\[
\begin{array}{l}
\{1,2\},\\
\{1,(a,d),(a,e),(a,f)\},\\
\{1,(b,d),(b,e),(b,f)\},\\
\{1,(c,d),(c,e),(c,f)\},\\
\{2,(a,d),(b,d),(c,d)\},\\
\{2,(a,e),(b,e),(c,e)\},\\
\{2,(a,f),(b,f),(c,f)\},\\
\{1\}\cup (\{a,b,c\}\times\{d,e,f\}),\\
\{2\}\cup (\{a,b,c\}\times\{d,e,f\}),\\
\{a,b,c\}\times\{d,e,f\}.
\end{array}
\]
We will refer to this hypergraph as $\cf_0$. It comes from Section 2
of \cite{CC}, which includes a construction of arbitrarily large {\rm
  $3$}-uniform hypergraphs on $n$ vertices with only
$\exp(O(\sqrt{\rule{0pt}{10pt}\log n}\,))$ distinct lines and no line consisting of all
$n$ vertices. All of these hypergraphs contain induced sub-hypergraphs isomorphic to 
$\cf_0$.

\medskip

\section{A generalization of  the De Bruijn - Erd\H{o}s theorem}\label{sec.no2}

Various generalizations of the De Bruijn - Erd\H{o}s theorem, or at
least of its first part, can be found in \cite{Bos, Rys, HP, FF84,
  Bab, Var, FF91, Sne, Ram, Cho, AMMV} and elsewhere. We offer a
generalization in a different spirit by strengthening
(\ref{no2no3dbe}): we will drop its assumption that no sub-hypergraph
induced by four vertices has three hyperedges. (As shown by the
hypergraph $\cf_0$ of the preceding section, the assumption that no
sub-hypergraph induced by four vertices has two hyperedges cannot be
dropped.)  This goes a long way towards generalizing the De Bruijn -
Erd\H{o}s theorem, but it does not quite get there: a description of
the extremal hypergraphs is also required. To provide this
description, we introduce additional notation and terminology.\\

We let $\binom{S}{3}$ denote the set of all three-point subsets of a
set $S$.  A {\em near-pencil\/} is a hypergraph $(V,\cl)$ such that
\[
\cl \;=\; \{V\setminus \{w\}\}\cup \{\{v,w\}: v\in V\setminus \{w\}\}\;\;\text{for some vertex $w$.}
\] 
We say that a $3$-uniform hypergraph $(V,\ce)$ {\em generates a
  near-pencil\/} if 
\[
\ce \;=\; \binom{V\setminus\{w\}}{3}\;\; \text{for some vertex $w$.}
\]
Clearly, this is the case if and only if the set $\cl$ of lines of
$(V,\ce)$ is such that $(V,\cl)$ is a near-pencil.

\medskip

A {\em finite projective plane\/} is a hypergraph $(V,\cl)$ in which,
for some integer $k$ greater than one, every two vertices belong to
precisely one hyperedge, $\abs{V}=k(k-1)+1$, and each hyperedge
contains precisely $k$ vertices.
We say that a $3$-uniform hypergraph $(V,\ce)$ {\em generates a
  finite projective plane\/} if, for some finite projective plane $(V,\cl)$, 
\[
\ce \;=\; \bigcup_{L\in\cl}\binom{L}{3}.
\]
Clearly, this is the case if and only if the set of lines of $(V,\ce)$
is $\cl$.

\medskip

The two extremal hypergraphs $(V,\cl)$ in the De Bruijn - Erd\H{o}s
theorem are exactly the near-pencil and the finite projective plane.

\medskip

We say that a $3$-uniform hypergraph $(V,\ce)$ is the {\em complement
  of a Steiner triple system\/} if every two of its vertices belong to
precisely one member of $\binom{V}{3}\setminus \ce$. Clearly, this is
the case if and only if the set of lines of $(V,\ce)$ is
$\{V\setminus\{x\}: x\in V\}$.\\

We use the graph-theoretic terminology and notation of Bondy and
Murty~\cite{BM}.  In particular,
\begin{itemize}
\item $P_n$ denotes the chordless path graph with $n$ vertices, 
\item $F+G$ denotes the disjoint union of graphs $F$ and $G$, 
\item $F\vee G$ denotes the {\em join\/} of graphs $F$ and $G$
  (defined as $F+G$ with additional edges that join every vertex of
  $F$ to every vertex of $G$).  
\end{itemize}
As usual, we call a graph {\em $F$-free\/} if it has no induced
subgraph isomorphic to graph $F$ and, when talking about sets, we use
the qualifier `maximal' as `maximal with respect to set-inclusion'
rather than as `largest'.

\begin{theorem}\label{thm.no2}
  Let $(V,\ce)$ be a $3$-uniform hypergraph on at least two vertices
  in which no four vertices induce two hyperedges and let $\cl$ be the
  set of lines of this hypergraph. If $V\not\in\cl$, then
  $\abs{\cl}\ge \abs{V}$, with equality if and only if $(V,\ce)$
  generates a near-pencil or a finite projective plane or is the
  complement of a Steiner triple system.
\end{theorem}

\begin{proof}
  Let $\ch$ denote the hypergraph and let $n$ denote the number of its
  vertices. We will use induction on $n$. The induction basis, $n=
  2$, is trivial; in the induction step, we distinguish between two
  cases. We may assume that $V\not\in\cl$.

\bigskip

{\sc Case 1:} {\em Every two vertices of $\cal H$ belong to precisely one maximal line.\/}\\
Let $\cl^{\max}$ denote the set of maximal lines of $\ch$. The De
Bruijn - Erd\H{o}s theorem guarantees that $\abs{\cl^{\max}}\ge n$,
with equality if and only if $(V,\cl^{\max})$ is a near-pencil or a
finite projective plane.  Since $\cl \supseteq \cl^{\max}$, we have
$\abs{\cl}\ge \abs{\cl^{\max}}\ge n$; if $\abs{\cl}=n$, then $\cl
=\cl^{\max}$, and so $(V,\cl)$ is a near-pencil or a finite projective
plane.\\

{\sc Case 2:} {\em Some two vertices of $\cal H$ belong to more than one maximal line.\/}\\
Let $p$ denote one of these two vertices and let
$\Sigma$ denote the graph with vertex set $V\setminus\{p\}$, where
vertices $u,v$ are adjacent if and only if $\{p,u,v\}\in\ce$. Since
$p,u,v,w$ do not induce two hyperedges,
\begin{subclaim} \label{claim.tripleSigma}
$u,v,w$ induce two edges in $\Sigma$ $\;\Ra \;\{u,v,w\}\in\ce$,\\
$u,v,w$ induce one$\,$ edge\phantom{s} in $\Sigma$ $\;\Ra \;\{u,v,w\}\not\in\ce$.
\end{subclaim}
A theorem of Seinsche~\cite{Sei} states that every connected
$P_4$-free graph with more than one vertex has a disconnected
complement; property \ref{claim.tripleSigma}. of $\Sigma$ guarantees that it is
$P_4$-free (if it contained an induced $P_4$, then the four
vertices of this $P_4$ would induce two hyperedges, a contradiction);
it follows that
\begin{subclaim} \label{claim.disconnected.complement}
every connected induced subgraph of $\Sigma$ with more than one vertex\\
has a disconnected complement.
\end{subclaim}
Having established \ref{claim.tripleSigma} and \ref{claim.disconnected.complement}, we distinguish between two subcases.\\

{\sc Subcase 2.1:} {\em $\Sigma$ is disconnected.\/}\\
In this subcase, we will prove that $\abs{\cl}>n$. To begin, the
assumption of this subcase means that
\[
\Sigma=\Sigma_1+ \Sigma_2+ \ldots + \Sigma_k, 
\]
where $k\ge 2$ and each $\Sigma_i$ is connnected; it follows from
\ref{claim.disconnected.complement} that each $\Sigma_i$ is either a
single vertex or has a disconnected complement.  For each $i=1,\ldots
,k$, let $V_i$ denote the vertex set of $\Sigma_i$ and let $W_i$
denote $V_i\cup\{p\}$.  We claim that
\begin{subclaim} \label{claim.Wi.line}
$x,y\in W_i$, $x\ne y$, $z\in V\setminus W_i$ $\;\Ra \;\{x,y,z\}\not\in\ce$.
\end{subclaim} 
When one of $x,y$ is $p$, the conclusion follows from the fact that
all vertices in $V_i$ are nonadjacent in $\Sigma$ to all vertices in
$V\setminus W_i$.  When $x,y$ are adjacent vertices of $\Sigma_i$, the
conclusion follows from the same fact, combined with
\ref{claim.tripleSigma}.  When $x,y$ are nonadjacent vertices of
$\Sigma_i$, consider a shortest path $P$ from $x$ to $y$ in
$\Sigma_i$. Since $x$ and $y$ are nonadjacent and $\Sigma$ is
$P_4$-free, $P$ has exactly three vertices.  Let $w$ be the unique
interior vertex of $P$. Now \ref{claim.tripleSigma} implies that
$\{x,y,w\}\in\ce$, $\{x,z,w\}\not\in\ce$, $\{y,z,w\}\not\in\ce$; in
turn, the fact that $x,y,z,w$ do not induce two hyperedges implies
that $\{x,y,z\}\not\in\ce$.

\medskip

Let $\ch_i$ denote the sub-hypergraph of $\ch$ induced by $W_i$. In the inductive argument, 
we shall use the following restatement of \ref{claim.Wi.line}:
\begin{subclaim} \label{claim.lines.Hi.conn} 
  $u,v\in W_i$, $u\ne v$ $\;\Ra \;$ the line $\ov{uv}$ in $\ch_i$
  equals the line $\ov{uv}$ in $\ch$.
\end{subclaim} 
Another way of stating \ref{claim.Wi.line} is 
\begin{subclaim} \label{claim.Wi.line.corr}
$u\in V_i$, $v\in V_j$, $i\ne j$ $\;\Ra \;\ov{uv}\cap (W_i\cup W_j)=\{u,v\}$. 
\end{subclaim} 
(The conclusion of \ref{claim.Wi.line.corr} can be strengthened to 
$\abs{\ov{uv}\cap W_s}\le 1$ for all $s$, but this is irrelevant to our argument.)

\medskip 

Next, let us show that 
\begin{subclaim} \label{claim.zwWj} $p\in \ov{uv}$
$\;\Ra \;\ov{uv}\subseteq W_i$ for some $i$.
\end{subclaim} 
Since $u$ and $v$ are distinct, we may assume that $u\neq p$, and so
$u\in V_i$ for some $i$. We claim that $v\in W_i$. If $v=p$, then this
is trivial; if $v\ne p$, then $p\in \ov{uv}$ implies that $u$ and $v$
are adjacent in $\Sigma$, and so $v\in V_i$. Now $u,v\in W_i$, and so
$\ov{uv}\subseteq W_i$ by \ref{claim.lines.Hi.conn}. 

\medskip

From \ref{claim.zwWj}, we will deduce that 
\begin{subclaim} \label{claim.extra1} 
$W_r\not\in \cl$ for some $r$.
\end{subclaim} 
By assumption, there is a vertex $y$ other than $p$ such that $p$ and
$y$ belong to at least two maximal lines of $\ch$; this vertex $y$
belongs to some $V_r$; by \ref{claim.zwWj}, every line containing both
$p$ and $y$ must be a subset of $W_r$; since at least two maximal
lines contain both $p$ and $y$, it follows that $W_r\not\in \cl$.

\medskip 

With $S$ standing for the set of subscripts $i$ such that $W_i$ is a
line of $\ch_i$, facts \ref{claim.lines.Hi.conn} and
\ref{claim.extra1} together show that$\abs{S}\le k-1$, and so we may
distinguish between the following three subcases:

\medskip

{\sc Subcase 2.1.1:} $\abs{S}=0$.\\ By the induction hypothesis, each
$\ch_i$ has at least $\abs{W_i}$ distinct lines; by
\ref{claim.lines.Hi.conn}, each of these lines is a line of $\ch$; since
$\abs{W_i\cap W_j}=1$ whenever $i\ne j$, all of these lines with
$i=1,\ldots ,k$ are distinct; it follows that
$\abs{\cl}\;\ge \;\sum_{i=1}^k\abs{W_i}\;=\;n+k-1\;>\;n$.

\medskip

{\sc Subcase 2.1.2:} $\abs{S}=1$.\\ We may assume that $S=\{1\}$. Now
$W_1$ is a line of $\ch$.  By \ref{claim.Wi.line.corr}, the
$\abs{V_1}\cdot\abs{V_2}$ lines $\ov{uv}$ of $\ch$ with $u\in V_1$,
$v\in V_2$ are all distinct; since they have nonempty intersections
with $V_2$, they are distinct from $W_1$. By the induction hypothesis,
each $\ch_i$ with $i\ge 2$ has at least $\abs{W_i}$ distinct lines; by
\ref{claim.lines.Hi.conn}, each of these lines is a line of $\ch$;
since $\abs{W_i\cap W_j}=1$ whenever $i\ne j$, all of these lines with
$i=2,\ldots ,k$ are distinct; since they are disjoint from $V_1$, they
are distinct from $W_1$ and from all $\ov{uv}$ with $u\in V_1$, $v\in
V_2$.  It follows that $\abs{\cl}\;\ge\;1+\sum_{i=2}^k\abs{W_i}+
\abs{V_1}\!\cdot\!\abs{V_2}\ge\;1+\sum_{i=2}^k\abs{W_i}+ \abs{V_1}
\;=\;n+k-1\;>\;n$.

\medskip

{\sc Subcase 2.1.3:} $2\le \abs{S}\le k-1$.\\
Let $W^\star$ denote $\bigcup_{i\in S}W_i$ and let $\ch^\star$ denote
the sub-hypergraph of $\ch$ induced by $W^\star$. From
\ref{claim.zwWj} and the assumption $\abs{S}\ge 2$, we deduce that no
line of $\ch$ contains $W^\star$. By the induction hypothesis,
$\ch^\star$ has at least $1+\sum_{i\in S}\abs{V_i}$ distinct lines; it
follows that $\ch$ has at least $1+\sum_{i\in S}\abs{V_i}$ distinct
lines $\ov{uv}$ with $u,v\in W^\star$. By the induction hypothesis,
each $\ch_i$ with $i\not\in S$ has at least $\abs{W_i}$ distinct
lines; by \ref{claim.lines.Hi.conn}, each of these lines is a line of
$\ch$; since $\abs{W_i\cap W_j}=1$ whenever $i\ne j$, all of these
lines with $i\not\in S$ are distinct; since they are disjoint from
$W^\star\setminus \{p\}$, they are distinct from all $\ov{uv}$ with
$u,v\in W^\star$.  It follows that $\abs{\cl}\;\ge\;1+\sum_{i\in
  S}\abs{V_i}+\sum_{i\not\in S}\abs{W_i}\;=\;n+(k-\abs{S})\;>\;n$.

\bigskip

{\sc Subcase 2.2:} {\em $\Sigma$ is connected.\/}\\
By \ref{claim.disconnected.complement}, the assumption of this subcase
implies that $\Sigma$ has a disconnected complement. This means that
\[
\Sigma=\Sigma_1\vee \Sigma_2\vee \ldots \vee \Sigma_k,
\]
where $k\ge 2$ and each $\Sigma_i$ has a connnected complement; it
follows from \ref{claim.disconnected.complement} that each $\Sigma_i$
is either a single vertex or a disconnected graph.  For each
$i=1,\ldots ,k$, let $V_i$ denote the vertex set of $\Sigma_i$ and let
$W_i$ denote $V_i\cup\{p\}$.  We claim that
\begin{subclaim} \label{claim.Wi.line.super}
$x,y\in W_i$, $x\ne y$, $z\in V\setminus W_i$ $\;\Ra \;\{x,y,z\}\in\ce$.
\end{subclaim} 
When one of $x,y$ is $p$, the conclusion follows from the fact that
all vertices in $V_i$ are adjacent in $\Sigma$ to all vertices in
$V\setminus W_i$.  When $x,y$ are nonadjacent vertices of $\Sigma_i$,
the conclusion follows from the same fact, combined with \ref{claim.tripleSigma}. When
$x,y$ are adjacent vertices of $\Sigma_i$, consider a shortest path
$P$ from $x$ to $y$ in the complement of $\Sigma_i$. Since $\Sigma_i$
is $P_4$-free, its complement is $P_4$-free; it follows that $P$ has
exactly three vertices. Let $w$ be the unique interior vertex of
$P$. Now \ref{claim.tripleSigma} implies that $\{x,y,w\}\not\in\ce$,
$\{x,z,w\}\in\ce$, $\{y,z,w\}\in\ce$; in turn, the fact that $x,y,z,w$
do not induce two hyperedges implies that $\{x,y,z\}\in\ce$.

\medskip

Let $\ch_i$ denote the sub-hypergraph of $\ch$ induced by $W_i$. In the inductive argument, 
we shall use the following restatement of \ref{claim.Wi.line.super}:
\begin{subclaim} \label{claim.lines.Hi.disconn} 
  $u,v\in W_i$, $u\ne v$ $\;\Ra \;$ the line $\ov{uv}$ in $\ch$ equals
  $Z\,\cup\, (V\setminus W_i)$, where $Z$ is the line $\ov{uv}$ in
  $\ch_i$.
\end{subclaim} 
Fact \ref{claim.lines.Hi.disconn} implies that 
\begin{subclaim} \label{claim.nonperverse} 
no line of $\ch_i$ equals $W_i$;
\end{subclaim}
in turn, the induction hypothesis applied to $\ch_i$ guarantees that
it has at least $\abs{W_i}$ distinct lines; now
\ref{claim.lines.Hi.disconn} implies that
\begin{subclaim} \label{claim.lines.Hi.extra} 
$\ch$ has at least $\abs{W_i}$ distinct lines $\ov{uv}$ with $u,v\in W_i$.
\end{subclaim} 

\medskip 

In addition, \ref{claim.lines.Hi.disconn} implies that 
\begin{subclaim} \label{claim.Wi.line.super.corr}
$u,v\in W_i$, $x,y\in W_j$, $i\ne j$, $\ov{uv}=\ov{xy}$  
$\;\;\Ra \;\;\ov{uv}=\ov{xy}=V\setminus \{p\}$.
\end{subclaim} 

\medskip

{\sc Subcase 2.2.1:} {\em $V\setminus \{p\}\not\in\cl$.\/} In this
subcase, \ref{claim.Wi.line.super.corr} guarantees that 
\[
u,v\in W_i,\; x,y\in W_j,\; i\ne j 
\;\Ra \;\ov{uv}\ne \ov{xy},
\]
and so \ref{claim.lines.Hi.extra} implies that 
$\abs{\cl}\;\ge\;\sum_{i=1}^k\abs{W_i}\;=\;n+k-1\;>\;n$.\\

{\sc Subcase 2.2.2:} {\em $V\setminus \{p\}\in \cl$.\/} Fact 
\ref{claim.lines.Hi.extra} guarantees that $\ch$ has at least
$\abs{W_i}-1$ distinct lines $\ov{uv}$ such that $u,v\in W_i$ and
$\ov{uv}\ne V\setminus \{p\}$, and so \ref{claim.Wi.line.super.corr},
combined with the assumption of this subcase, implies that
$\abs{\cl}\;\ge\; \sum_{i=1}^k(\abs{W_i}-1)+1\;=\;n$.

\medskip

To complete the analysis of this subcase, let us consider its extremal
hypergraphs, those with $\abs{\cl}=n$.  Here, 
\begin{subclaim} \label{claim.extreme} each $\ch_i$ has precisely
  $\abs{W_i}$ distinct lines and $V_i$ is one of these lines;\\
  $\cl$ consists of all the sets $Z\,\cup\, (V\setminus W_i)$ such that $Z$ is a line of some $\ch_i$.
\end{subclaim}
We are going to prove that
\begin{subclaim} \label{claim.Hi.near.pencil}
the hyperedge set $\ce_i$ of each $\ch_i$ is $\binom{V_i}{3}$
\end{subclaim}
Since $V_i$ is a line of $\ch_i$, it has at least two vertices.  If
$|V_i| = 2$, then both \ref{claim.Hi.near.pencil} and
\ref{claim.nonperverse} amount to saying that $\ch_i$ has no
hyperedges.  Now we will assume that $|V_i| \ge 3$.  The induction
hypothesis, combined with \ref{claim.nonperverse}, guarantees that
$\ch_i$ generates a near-pencil or a finite projective plane or is the
complement of a Steiner triple system; since $V_i$ is one of the lines
of $\ch_i$, proving \ref{claim.Hi.near.pencil} amounts to proving that
$\ch_i$ generates a near-pencil. The possibility of $\ch_i$ generating
a finite projective plane is excluded by the fact that one of the
lines of $\ch_i$ (namely, $V_i$) includes all the vertices but one.
The possibility of $\ch_i$ being the complement of a Steiner triple
system is excluded by the fact that $\Sigma_i$ is disconnected, and so
it includes vertices $u,v,w$ such that $u$ is nonadjacent to both $v,
w$: now $u$ and $p$ belong to at least two members of
$\binom{W_i}{3}\setminus \ce_i$ (namely, $\{u,v,p\}$ and $\{u,w,p\}$).
This completes our proof of \ref{claim.Hi.near.pencil}.

\medskip

Next, let us prove that 
\begin{subclaim} \label{claim.new1} for every $i=1,2,\ldots ,k$ and
  every $x$ in $V_i$, there is an $L$ in $\cl$ such that $V_i\setminus
  L=\{x\}$.
\end{subclaim}
Choose any vertex $z$ in $V\setminus W_i$. Since $\ov{xz}\ne V$, there
is a $w$ in $V$ such that $w\ne x$, $w\ne z$, and
$\{x,z,w\}\not\in\ce$; fact \ref{claim.Wi.line.super} implies that
$w\not\in W_i$. Next, consider an arbitrary vertex $y$ in
$V_i\setminus \{x\}$. Fact \ref{claim.Wi.line.super} guarantees that
$\{x,y,z\}\in\ce$ and $\{x,y,w\}\in\ce$; in turn, the fact that
$x,y,z,w$ do not induce two hyperedges implies that $\{y,z,w\}\in\ce$.
We conclude that $y\in\ov{zw}$, and so $V_i\setminus \ov{zw}=\{x\}$,
which completes our proof of \ref{claim.new1}. 

\medskip

Finally, let us prove that 
\begin{subclaim} \label{claim.new2} 
$V\setminus \{x\}\in \cl$ for all $x$ in $V$.
\end{subclaim}
Since $V\setminus \{p\}\in \cl$ by assumption of this subcase, we may
restrict our argument to vertices $x$ distinct from $p$. Every such
$x$ belongs to some $V_i$ and, by \ref{claim.new1}, there is an $L$ in
$\cl$ such that $V_i\setminus L=\{x\}$; by \ref{claim.extreme}, there
are a subscript $j$ and a line $Z$ of $\ch_j$ such that
$L\;=\;Z\,\cup\, (V\setminus W_j)$. Now $V_i\not\subseteq L$ and
$V_r\subseteq L$ whenever $r\ne j$, and so $j=i$. By
\ref{claim.Hi.near.pencil}, every line of $\ch_i$ either equals $V_i$
or includes $p$; since $V_i\setminus Z=V_i\setminus L=\{x\}$, it
follows that $p\in Z$. Since $V_i\setminus Z=\{x\}$ and $p\in Z$
together imply that $Z=W_i\setminus\{x\}$, we conclude that
$L=V\setminus \{x\}$. This completes our proof of \ref{claim.new2}.

\medskip

Since $\abs{\cl}=n$, fact \ref{claim.new2} guarantees that $\cl$
consists of the $n$ sets $V\setminus \{x\}$ with $x$ ranging over
$V$. This means that for every two vertices $u$ and $v$, there is a
unique vertex in $V\setminus\ov{uv}$, which is just another way of
saying that $\ch$ is the complement of a Steiner triple system.
\end{proof}

\section{Metric and pseudometric hypergraphs}\label{sec.met}

We say that a $3$-uniform hypergraph $(V,\ce)$ is {\em metric\/}
if there is a metric space $(V,\dist)$ such that 
\[
\mbox{ $\ce= \{\{u,v,w\}: u,v,w$ are all distinct and 
$\dist(u,v)+\dist(v,w)=\dist(u,w)\}$. }
\]
Chen and Chv\' atal~\cite{CC} asked whether or not all metric
hypergraphs have the De Bruijn - Erd\H os property; this question was
investigated further by Chiniforooshan and Chv\'{a}tal~\cite{CC2}. 

\medskip

All induced sub-hypergraphs of metric hypergraphs are metric, and so
metric hypergraphs can be characterized as hypergraphs without certain
induced sub-hypergraphs, namely, the minimal non-metric ones.  If
there are only finitely many minimal non-metric hypergraphs, then
metric hypergraphs can be recognized in polynomial time.  However, it
is conceivable that there are infinitely many minimal non-metric
hypergraphs and it is not clear whether metric hypergraphs can be
recognized in polynomial time. 

\medskip 

In this section, we will list three minimal non-metric hypergraphs. To
begin, we will prove that the hypergraphs without the De Bruijn -
Erd\H os property mentioned in Section~\ref{sec.intro} cannot provide
a negative answer to the Chen--Chv\' atal question. All of these
hypergraphs contain the $11$-vertex hypergraph denoted $\cf_0$ in
Section~\ref{sec.intro}. We will prove that $\cf_0$ is not metric. In
fact, we will prove that it contains an $8$-vertex induced
sub-hypergraph $\cf_1$, which is minimal non-metric. The vertex set of
$\cf_1$ is $\{1,2\}\cup (\{a,b,c\}\times\{d,e\})$; its hyperedges are
the $\binom{6}{3}$ three-point subsets of $\{a,b,c\}\times\{d,e\}$ and
the nine three-point sets $\{i,(x_1,x_2),(y_1,y_2)\}$ with $x_i=y_i$.

\medskip

We will also prove that no complement of a Steiner triple system with
more than three vertices is metric.  In fact, we will exhibit
$6$-vertex minimal non-metric hypergraphs $\cf_2$ and $\cf_3$ such
that every complement of a Steiner triple system with more than three
vertices contains at lest one of $\cf_2$ and $\cf_3$.

\medskip

A ternary relation $\cb$ on a set $V$ is called a {\em metric betweenness\/} 
if there is a metric $\dist$ on $V$ such that $(u,v,w)\in \cb$ if and only if 
\[
\mbox{ $u,v,w$ are all distinct and 
$\dist(u,v)+\dist(v,w)=\dist(u,w)$. }
\]
Menger \cite{Men} seems to have been the first to study this
relation. He proved that, in addition to the obvious properties 
\begin{tabbing}
xx\=(M12)x\=\kill
\>(M0)\>if $(u,v,w)\in \cb$, then $u,v,w$ are three 
 points,\\
\>(M1)\>if $(u,v,w)\in \cb$, then $(w,v,u)\in \cb$,\\
\>(M2)\>if $(u,v,w)\in \cb$, then $(u,w,v)\not\in \cb$,
\end{tabbing}
every metric betweenness $\cb$ has the property 
\begin{tabbing}
xx\=(M12)x\=\kill
\>(M3)\>if $(u,v,w),\,(u,w,x)\in \cb$, then
$(u,v,x),\,(v,w,x)\in \cb$.
\end{tabbing}
We will call a ternary relation $\cb$ on a set $V$ a {\em
  pseudometric betweenness\/} if it has properties (M0), (M1), (M2),
(M3). Not every pseudometric betweenness is a metric betweenness:
see~\cite{Chv} for more on this subject.

\medskip

Every ternary relation $\cb$ on a set $V$ that has property (M0) gives
rise to a hypergraph $(V,\ce(\cb))$ by discarding the order on each
triple in $\cb$:
\[
\ce(\cb)=\{\{u,v,w\}:\,(u,v,w)\in \cb\}.
\]
We will say that a $3$-uniform hypergraph $(V,\ce)$ is {\em
  pseudometric\/} if there is a pseudometric betweenness $\cb$ on $V$
such that $\ce=\ce(\cb)$. Every metric hypergraph is pseudometric, but
the converse is false: the {\em Fano hypergraph\/} is pseudometric but
it is not metric.  (This hypergraph has seven vertices and seven
hyperedges, every two of which share a single vertex; like all
$3$-uniform hypergraphs in which no two hyperedges share two vertices,
it is pseudometric; it has been proved \cite{Chv,Che} that it is not
metric, but neither of the two proofs is very short.)  We will prove
that $\cf_1$, $\cf_2$, $\cf_3$ are not even pseudometric. (There are
many other minimal non-pseudometric hypergraphs: our computer search
revealed 113 non-isomorphic ones on six vertices.)

\begin{question}
  True or false? All pseudometric hypergraphs have the De Bruijn -
  Erd\H os property.
\end{question}

\medskip

In proving that $\cf_1$ is not pseudometric, we shall rely on the
following fact.
\begin{lemma}\label{lem.trique}
  If $\cb$ is a pseudometric betweenness on a set $V$ such that
  $\ce(\cb)=\binom{V}{3}$ and $\abs{V}\ge 5$, then there is an
  injection $f:V\ra {\bf R}$ such that $(x,y,z)\in \cb$ if and only if
  $f(y)$ is between $f(x)$ and $f(z)$.
\end{lemma}

\begin{proof} We will use induction on $\abs{V}$. 
To begin, we claim that 
\begin{itemize}
\item[(i)] for some element $p$ of $V$, the elements of
  $V\setminus\{p\}$ can be enumerated as $v_1$, $v_2$, \ldots ,
  $v_{n-1}$, in such a way that $(v_i,v_j,v_k)\in \cb$ if and only if
  $j$ is between $i$ and $k$.
\end{itemize}
To justify this claim, we consider the case of $\abs{V}=5$ separately
from the rest. Here, note that $\tbinom{5}{3}$ is not a multiple of
$3$, and so some $a$ and $b$ appear in one or two triples of the form
$(a,x,b)$ in $\cb$. This means that there are $a,b,c,d$ such that
$(a,c,b)\in\cb$ and $(a,d,b)\not\in\cb$. Since $\{a,d,b\}\in\ce(\cb)$,
we must have $(a,b,d)\in\cb$ or $(b,a,d)\in\cb$. Setting $v_1=a$,
$v_2=c$, $v_3=b$, $v_4=d$ if $(a,b,d)\in\cb$ and $v_1=b$, $v_2=c$,
$v_3=a$, $v_4=d$ if $(b,a,d)\in\cb$, we get $(v_1,v_2,v_3),
(v_1,v_3,v_4)\in\cb$; now (M3) with $u=v_1$, $v=v_2$, $w=v_3$, $x=v_4$
guarantees that $(v_1,v_2,v_4), (v_2,v_3,v_4)\in\cb$. In the case of
$\abs{V}\ge 6$, claim (i) is just the induction hypothesis.

\medskip

With (i) justified, we distinguish between two cases.

\medskip

{\sc Case 1:} {\em $(v_i,p,v_{i+1})\in \cb$ for some $i$.}\\
In this case, we claim that the proof can be completed by setting
$f(p)=i+0.5$ and $f(v_j)=j$ for all $j$. To justify this claim, we
first use induction on $j$, with the basis at $j=i+1$ and (M3) applied
to $(v_i,p,v_j)$, $(v_i,v_j,v_{j+1})$ in the induction step, to show
that $(v_i,p,v_j)\in \cb$ for all $j=i+1,i+2,\ldots , n-1$. In turn,
(M3) applied to $(v_i,p,v_j)$, $(v_i,v_j,v_k)$ shows that
$(p,v_j,v_k)\in \cb$ whenever $i+1\le j<k\le n-1$. Appealing to the
flip symmetry of the sequence $v_1$, $v_2$, \ldots, $v_{n-1}$, we also
note that $(v_r,v_s,p)\in \cb$ whenever $1\le r<s\le i$. Finally,
given any $r$ and $j$ such that $1\le r<i$ and $i+1\le j\le n-1$, we
apply (M3) to $(v_i,p,v_j)$, $(v_r,v_i,v_j)$ in order to check that
$(v_r,p,v_j)\in \cb$. This completes our analysis of Case 1.

\medskip

{\sc Case 2:} {\em For each $i=1,\ldots , n-2$, we have $(p,v_i,v_{i+1})\in \cb$ or $(v_i,v_{i+1},p)\in \cb$.}\\
In this case, we claim that the proof can be completed by setting
$f(v_j)=j$ for all $j$ and either $f(p)=0$ or $f(p)=n$. To justify
this claim, we will first prove that
\begin{itemize}
\item[(ii)] there is no $i$ such that $2\le i\le n-2$ and $(v_{i-1},v_i,p), (p,v_i,v_{i+1})\in \cb$,
\item[(iii)] there is no $i$ such that $2\le i\le n-2$ and $(p,v_{i-1},v_i), (v_i,v_{i+1},p)\in \cb$:
\end{itemize}

To justify (ii), assume the contrary.  Since
$\{v_{i-1},p,v_{i+1}\}$ belongs to $\ce(\cb)$, we may label its elements as $x_1,x_2,x_3$ in such a way that 
$(x_1,x_2,x_3)\in\cb$. Since 
\[(v_{i-1},v_i,p),(p,v_i,v_{i+1}), (v_{i+1},v_i,v_{i-1})\in\cb,\] we
have $(x_1,v_i,x_2),(x_2,v_i,x_3),(x_3,v_i,x_1)\in\cb$; now (M3) with
$u=x_1$, $v=v_i$, $w=x_2$, $x=x_3$ implies $(v_i,x_2,x_3)\in\cb$,
which, together with $(x_2,v_i,x_3)\in\cb$, contradicts (M2).

\medskip

To justify (iii), assume the contrary. Since $n>4$, we have $i>2$ or
$i<n-2$ or both; symmetry lets us assume that $i<n-2$. Now (ii) with
$i+1$ in place of $i$ guarantees that $(p,v_{i+1},v_{i+2})\not\in\cb$;
the assumption of this case guarantees that
$(v_{i+1},p,v_{i+2})\not\in\cb$; it follows that
$(v_{i+1},v_{i+2},p)\in\cb$.  There are three ways of including
$\{p,v_{i-1},v_{i+1}\}$ in $\ce(\cb)$; we will show that each of them
leads to a contradiction.  If $(p,v_{i-1},v_{i+1})\in\cb$, then (M3)
with $(p,v_{i+1},v_i)$ implies $(v_{i-1},v_{i+1},v_i)\in\cb$,
contradicting $(v_{i-1},v_i,v_{i+1})\in\cb$.  If
$(v_{i+1},p,v_{i-1})\in\cb$, then (M3) with $(v_{i+1},v_{i+2},p)$
implies $(v_{i+1},v_{i+2},v_{i-1})\in\cb$, contradicting
$(v_{i-1},v_{i+1},v_{i+2})\in\cb$.  If $(p,v_{i+1},v_{i-1})\in\cb$,
then (M3) with $(p,v_{i-1},v_i)$ implies
$(v_{i+1},v_{i-1},v_i)\in\cb$, contradicting
$(v_{i-1},v_i,v_{i+1})\in\cb$.

\medskip

Claims (ii) and (iii), combined with the assumption of this case, imply
that we have either $(p,v_i,v_{i+1})\in \cb$ for all $i=1,\ldots , n-2$
or else $(v_{i-1},v_i,p)\in \cb$ for all $i=2,\ldots , n-1$.  In the
first case, induction on $d$ with the basis at $d=1$ and (M3) applied
to $(p,v_i,v_{i+d})$, $(p,v_{i+d},v_{i+d+1})$ in the induction step
shows that $(p,v_i,v_{i+d})\in\cb$ for all $d=1,\ldots , n-1-i$; it
follows that we may set $f(p)=0$ and $f(v_j)=j$ for all $j$.  In the
second case, induction on $d$ with the basis at $d=1$ and (M3) applied
to $(v_{i-d},v_i,p)$, $(v_{i-(d+1)},v_{i-d},p)$ in the induction step
shows that $(v_{i-d},v_i,p)\in\cb$ for all $d=1,\ldots , i-1$; it
follows that we may set $f(v_j)=j$ for all $j$ and $f(p)=n$.
This completes our analysis of Case 2.
\end{proof}

A weaker version of Lemma~\ref{lem.trique} was proved by Richmond and
Richmond \cite{RR} and later also by Dovgoshei and Dordovskii
\cite{DoDo}: there, the assumption that $\cb$ is pseudometric is
replaced by the stronger assumption that $\cb$ is metric. As noted
in~\cite{DoDo}, this weaker version of Lemma~\ref{lem.trique} implies
a special case ($d=1$ and finite spaces) of the following result of
Menger (\cite{Men}, Satz 1): {\em If every $(d+3)$-point subspace of a
  metric space admits an isometric embedding into ${\bf R}^d$, then
  the whole space admits an isometric embedding into ${\bf R}^d$.\/}

\medskip

The conclusion of Lemma~\ref{lem.trique} may fail when
$\abs{V}=4$: here, if $\cb$ includes two triples of the form $(u,v,w),\,(u,w,x)$, then
(M3) implies that it is isomorphic to 
\[
\{(a,b,c),(c,b,a),\; (a,b,d),(d,b,a),\; (a,c,d),(d,c,a),\; (b,c,d),(d,c,b)\}
\]
as in the lemma's conclusion, but $\cb$ may include no such triples,
in which case it is isomorphic to
\[
\{(a,b,c),(c,b,a),\; (b,c,d),(d,c,b),\; (c,d,a),(a,d,c),\; (d,a,b),(b,a,d)\}.
\]
This has been also pointed out (again, with ``pseudometric'' replaced by
``metric'') by Richmond and Richmond \cite{RR} and by Dovgoshei and
Dordovskii \cite{DoDo}.

\begin{theorem}\label{thm.f1}
$\cf_1$ is a minimal non-metric hypergraph. 
\end{theorem}
\begin{proof}
  We will prove that $\cf_1$ is not pseudometric and that all its proper
  induced sub-hypergraphs are metric.

\medskip

Recall that the vertex set of $\cf_1$ is $\{1,2\}\cup
(\{a,b,c\}\times\{d,e\})$; its hyperedges are the $\binom{6}{3}$
three-point subsets of $\{a,b,c\}\times\{d,e\}$ and the nine
three-point sets $\{i,(x_1,x_2),(y_1,y_2)\}$ with $x_i=y_i$.  For each
$r$ in $\{a,b,c\}\times \{d,e\}$ and for each $i=1,2$, let $r_i$
denote the $i$-th component of $r$. We claim that
\begin{itemize}
\item[($\star$)] 
{\em For every injection $f:\{a,b,c\}\times \{d,e\}\ra {\bf R}$,\\
there exist $r,s,t,u$ in $\{a,b,c\}\times \{d,e\}$ such that
\[
\begin{array}{l}
\mbox{$r,s,t,u$ are four  vertices,}\\ 
\mbox{$f(s)$ is between $f(r)$ and $f(t)$,}\\ 
\mbox{$f(u)$ is not between $f(r)$ and $f(t)$, and}\\ 
\mbox{either $r_1=t_1$ or else $r_2=t_2$, $s_2=u_2$.}  
\end{array}
\hspace{3cm}
\]
}
\end{itemize} 
To verify this claim, we may assume without loss of generality that
the range of $f$ is $\{0,1,2,3,4,5\}$. If there are distinct $r,t$
with $r_1=t_1$ and $2\le \abs{f(r)-f(t)}\le 4$, then
($\star$) can be satisfied by these $r,t$ and a suitable
choice of $s,u$. If there are distinct $r,t$ with $r_2=t_2$ and
$\abs{f(r)-f(t)}=3$, then ($\star$) can be satisfied by these
$r,t$ and a suitable choice of $s,u$. If neither of these two
conditions is met, then
\begin{eqnarray*}
r\ne t,\, r_1=t_1 &\Ra  & \abs{f(r)-f(t)}\in \{1,5\},\\
r\ne t,\, r_2=t_2 &\Ra  & \abs{f(r)-f(t)}\in \{1,2,4,5\},
\end{eqnarray*}
in which case $x_2=y_2 \;\Ra  \; f(x)\equiv f(y) \pmod {2}$. But then
($\star$) can be satisfied by any choice of distinct $r,t$
with $r_2=t_2$ and a suitable choice of $s,u$.\\ 

To prove that $\cf_1$ is not pseudometric, assume the contrary: there
is a pseudometric betweenness $\cb$ such that $\ce(\cb)$ is the
hyperedge set of $\cf_1$. Now all $3$-point subsets of
$\{a,b,c\}\times \{d,e\}$ belong to $\ce({\cal B})$, and so
Lemma~\ref{lem.trique} guarantees the existence of an injection
$f:\{a,b,c\}\times \{d,e\}\ra {\bf R}$ such that $(x,y,z)\in \cb$ if
and only if $f(y)$ is between $f(x)$ and $f(z)$. Next, ($\star$)
implies that there are distinct $r,s,t,u$ in $\{a,b,c\}\times \{d,e\}$
such that \[(r,s,t)\in \cb,\] $(r,u,t)\not\in \cb$, and either
$r_1=t_1$ or else $r_2=t_2$, $s_2=u_2$.  Since $(r,u,t)\not\in \cb$
and $\{r,u,t\}\in \ce(\cb)$, we have $(r,t,u)\in \cb$ or $(t,r,u)\in
\cb$; after switching $r$ and $t$ if necessary, we may assume that
\[
(r,t,u)\in \cb.
\]
Writing $x=1$ if $r_1=t_1$ and $x=2$ if $r_2=t_2$, note that
\[
\{x,r,t\}\in \ce(\cb),\;\{x,r,s\}\not\in \ce(\cb),\;\{x,r,u\}\not\in \ce(\cb).
\]

Since $\{x,r,t\}\in \ce(\cb)$, we may distinguish between three cases.\\
 
In case $(x,r,t)\in \cb$, property (M3) and $(r,s,t)\in \cb$ imply
$(x,r,s)\in \cb$, contradicting $\{x,r,s\}\not\in \ce(\cb)$.

\medskip

In case $(r,t,x)\in \cb$, property (M3) and $(r,s,t)\in \cb$ imply
$(r,s,x)\in \cb$, contradicting $\{x,r,s\}\not\in \ce(\cb)$.

\medskip

In case $(r,x,t)\in \cb$, property (M3) and $(r,t,u)\in \cb$ imply
$(r,x,u)\in \cb$, contradicting $\{x,r,u\}\not\in \ce(\cb)$.\\

Symmetry of $\cf_1$ reduces checking that all its proper induced
sub-hypergraphs are metric to checking just three cases: vertex $1$
removed, vertex $2$ removed, and a vertex in $\{a,b,c\}\times\{d,e\}$
removed. Here are the distance functions of the corresponding three
metric spaces:
\[
\phantom{ c,d)}\cf_1\setminus 1:\;
\begin{array}{|c|c|c|c|c|c|c|c|}\hline
      & \!(a,d)\! & \!(b,d)\! & \!(c,d)\! & \!(a,e)\! & \!(b,e)\! & \!(c,e)\! & \;\;2\;\;\\\hline
\;(a,d)\; & \;0\; & \;1\; & \;2\; & \;3\; & \;4\; & \;5\; & \;3\;\\\hline
\;(b,d)\; & \;1\; & \;0\; & \;1\; & \;2\; & \;3\; & \;4\; & \;2\;\\\hline
\;(c,d)\; & \;2\; & \;1\; & \;0\; & \;1\; & \;2\; & \;3\; & \;1\;\\\hline
\;(a,e)\; & \;3\; & \;2\; & \;1\; & \;0\; & \;1\; & \;2\; & \;1\;\\\hline
\;(b,e)\; & \;4\; & \;3\; & \;2\; & \;1\; & \;0\; & \;1\; & \;2\;\\\hline
\!(c,e)\! & \;5\; & \;4\; & \;3\; & \;2\; & \;1\; & \;0\; & \;3\;\\\hline
\;2\; & \;3\; & \;2\; & \;1\; & \;1\; & \;2\; & \;3\; & \;0\;\\ \hline
\end{array}
\]
\smallskip
\[
\phantom{ c,d)}\cf_1\setminus 2:\;
\begin{array}{|c|c|c|c|c|c|c|c|}\hline
      & \!(a,d)\! & \!(a,e)\! & \!(b,d)\! & \!(b,e)\! & \!(c,d)\! & \!(c,e)\! & \;\;1\;\;\\\hline
\;(a,d)\; & \;0\; & \;2\; & \;4\; & \;6\; & \;8\; & 10    & \;6\;\\\hline
\;(a,e)\; & \;2\; & \;0\; & \;2\; & \;4\; & \;6\; & \;8\; & \;4\;\\\hline
\;(b,d)\; & \;4\; & \;2\; & \;0\; & \;2\; & \;4\; & \;6\; & \;5\;\\\hline
\;(b,e)\; & \;6\; & \;4\; & \;2\; & \;0\; & \;2\; & \;4\; & \;3\;\\\hline
\;(c,d)\; & \;8\; & \;6\; & \;4\; & \;2\; & \;0\; & \;2\; & \;4\;\\\hline
\!(c,e)\! & 10    & \;8\; & \;6\; & \;4\; & \;2\; & \;0\; & \;6\;\\\hline
\;1\; & \;6\; & \;4\; & \;5\; & \;3\; & \;4\; & \;6\; & \;0\;\\ \hline
\end{array}
\]
\smallskip
\[
\,\cf_1\setminus (c,d):\;
\begin{array}{|c|c|c|c|c|c|c|c|}\hline
      & \!(a,d)\! & \!(b,d)\! & \!(b,e)\! & \!(c,e)\! & \!(a,e)\! & \;\;1\;\; & \;\;2\;\;\\\hline
\;(a,d)\; & \;0\; & \;2\; & \;4\; & \;6\; & \;8\; & \;5\; & \;4\;\\\hline
\;(b,d)\; & \;2\; & \;0\; & \;2\; & \;4\; & \;6\; & \;5\; & \;2\;\\\hline
\;(b,e)\; & \;4\; & \;2\; & \;0\; & \;2\; & \;4\; & \;3\; & \;2\;\\\hline
\;(c,e)\; & \;6\; & \;4\; & \;2\; & \;0\; & \;2\; & \;3\; & \;4\;\\\hline
\;(a,e)\; & \;8\; & \;6\; & \;4\; & \;2\; & \;0\; & \;3\; & \;6\;\\\hline
\;1\; & \;5\; & \;5\; & \;3\; & \;3\; & \;3\; & \;0\; & \;4\;\\\hline
\;2\; & \;4\; & \;2\; & \;2\; & \;4\; & \;6\; & \;4\; & \;0\;\\ \hline
\end{array}
\]
\end{proof}

Next, we will consider the hypergraphs $\cf_2$ and $\cf_3$ defined by
$\cf_2=(V,\ce_2)$ and $\cf_3=(V,\ce_3)$, where
\begin{align*}
& V = \{a_1,b_1,a_2,b_2,a_3,b_3\},\\
& \ce_2 = \tbinom{V}{3}\setminus \{\{a_1,b_2,b_3\},\{a_2,b_1,b_3\},\{a_3,b_1,b_2\}, \{a_1,a_2,a_3\}\},\\
& \ce_3 = \tbinom{V}{3}\setminus \{\{a_1,b_2,b_3\},\{a_2,b_1,b_3\},\{a_3,b_1,b_2\}\}.
\end{align*}

\begin{theorem}\label{23}
$\cf_2$ and $\cf_3$ are minimal non-metric hypergraphs. 
\end{theorem}
\begin{proof}
  We will prove that neither of $\cf_2$ and $\cf_3$ is pseudometric
  and that all their proper induced sub-hypergraphs are metric.

\medskip

To prove that neither of $\cf_2$ and $\cf_3$ is pseudometric, assume
the contrary: some pseudometric betweenness $\cb$ on
$\{a_1,b_1,a_2,b_2,a_3,b_3\}$ has
\begin{multline*}
\tbinom{V}{3}\setminus \{\{a_1,b_2,b_3\},\{a_2,b_1,b_3\},\{a_3,b_1,b_2\},\{a_1,a_2,a_3\}\}\subseteq \ce(\cb)\\ 
\subseteq \tbinom{V}{3}\setminus \{\{a_1,b_2,b_3\},\{a_2,b_1,b_3\},\{a_3,b_1,b_2\}\}.
\end{multline*}
Since $\{b_1,b_2,b_3\}\in \ce(\cb)$, we may assume (after permuting the subscripts if necessary) that
\begin{itemize}
\item[(i)] $(b_1,b_2,b_3)\in \cb$.
\end{itemize}
Now, since $\{a_1,b_2,b_3\}\not\in \ce(\cb)$, property (M3) implies
that $(b_1,b_3,a_1)\not\in \cb$ and $(b_3,b_1,a_1)\not\in \cb$; since 
$\{a_1,b_1,b_3\}\in \ce(\cb)$, it follows that
\begin{itemize}
\item[] $(b_1,a_1,b_3)\in \cb$.
\end{itemize}
Next, since $\{a_1,b_2,b_3\}\not\in \ce(\cb)$, property (M3) implies
that $(b_1,a_1,b_2)\not\in \cb$ and $(b_1,b_2,a_1)\not\in \cb$; since 
$\{a_1,b_1,b_2\}\in \ce(\cb)$, it follows that
\begin{itemize}
\item[] $(a_1,b_1,b_2)\in \cb$.
\end{itemize}
Finally, since $\{a_3,b_1,b_2\}\not\in \ce(\cb)$, property (M3) implies
that $(a_1,b_2,a_3)\not\in \cb$ and $(b_2,a_1,a_3)\not\in \cb$; since 
$\{a_1,b_2,b_3\}\in \ce(\cb)$, it follows that
\begin{itemize}
\item[(ii)] $(b_2,a_3,a_1)\in \cb$.
\end{itemize}
Switching subscripts $1$ and $3$ in this derivation of (ii) from (i),
we observe that (i) also implies
\begin{itemize}
\item[(iii)] $(b_2,a_1,a_3)\in \cb$.
\end{itemize}
But (ii) and (iii) together contradict property (M2).

\medskip

Symmetry reduces checking that all proper induced sub-hypergraphs of
$\cf_2$ and $\cf_3$ are metric to checking that two hypergraphs are
metric: $\cf_3\setminus a_1$ (isomorphic to $\cf_3\setminus a_2$ ,
$\cf_3\setminus a_3$, and to all five-point induced sub-hypergraphs of
$\cf_2$) and $\cf_3\setminus b_1$ (isomorphic to $\cf_3\setminus b_2$
and $\cf_3\setminus b_3$). Here are distance functions certifying that
these two hypergraphs are metric:
\[
\cf_3\setminus a_1:\;\;
\begin{array}{|c|c|c|c|c|c|c|c|}\hline
        &\!a_2\!&\!a_3\!&\!b_3\!&\!b_2\!&\!b_1\!\\\hline
\;a_2\; & \;0\; & \;1\; & \;2\; & \;1\; & \;2\; \\\hline
\;a_3\; & \;1\; & \;0\; & \;1\; & \;2\; & \;3\; \\\hline
\;b_3\; & \;2\; & \;1\; & \;0\; & \;1\; & \;2\; \\\hline
\;b_2\; & \;1\; & \;2\; & \;1\; & \;0\; & \;3\; \\\hline
\;b_1\; & \;2\; & \;3\; & \;2\; & \;3\; & \;0\; \\\hline
\end{array}
\]
\smallskip
\[
\cf_3\setminus b_1:\;\;
\begin{array}{|c|c|c|c|c|c|c|c|}\hline
        &\!a_2\!&\!a_1\!&\!a_3\!&\!b_3\!&\!b_2\!\\\hline
\;a_2\; & \;0\; & \;1\; & \;2\; & \;1\; & \;1\; \\\hline
\;a_1\; & \;1\; & \;0\; & \;1\; & \;2\; & \;2\; \\\hline
\;a_3\; & \;2\; & \;1\; & \;0\; & \;1\; & \;1\; \\\hline
\;b_3\; & \;1\; & \;2\; & \;1\; & \;0\; & \;2\; \\\hline
\;b_2\; & \;1\; & \;2\; & \;1\; & \;2\; & \;0\; \\\hline
\end{array}
\]
\end{proof}

\begin{corollary}
No complement of a Steiner triple system with more than three vertices is metric.
\end{corollary}

\begin{proof} We will point out that every complement $(V,\ce)$ of a
  Steiner triple system with more than three vertices contains at
  least one of $\cf_2$ and $\cf_3$. To do this, note that, since $V$
  includes more than three vertices, it includes pairwise distinct
  vertices $b_1,b_2,b_3$ such that $\{b_1,b_2,b_3\}\in \ce$. Since
  every two vertices in $V$ belong to precisely one member of
  $\binom{V}{3}\setminus \ce$, it follows first that there are
  vertices $a_1,a_2,a_3$ such that
  $\{a_1,b_2,b_3\},\{a_2,b_1,b_3\},\{a_3,b_1,b_2\}\not\in \ce$, then
  that $a_1,a_2,a_3,b_1,b_2,b_3$ are six distinct vertices, and
  finally that these six vertices induce in $(V,\ce)$ one of $\cf_2$
  and $\cf_3$.
\end{proof}

\section{Variations}\label{sec.var}

In this section, we prove two variations on Theorem~\ref{thm.no2}.

\begin{theorem}\label{thm.no1no3}
  If, in a $3$-uniform hypergraph, no sub-hypergraph induced by four
   vertices has one or three hyperedges, then the
  hypergraph has the De Bruijn - Erd\H{o}s property.
\end{theorem}

\begin{proof}
  Let $(V,\ce)$ be a $3$-uniform hypergraph in which no four
   vertices induce one or three hyperedges. We claim that
\begin{itemize}
\item[($\star$)] $\ov{uv}=\ov{uw}$ and $v\ne w
\;\Ra  \; \ov{vw}=V$.
\end{itemize}
To justify this claim, consider an arbitrary vertex $x$ other than
$u,v,w$: we propose to show that $x\in \ov{vw}$. Since
$\ov{uv}=\ov{uw}$, we have $\{u,v,w\}\in\ce$, and so the four vertices
$u,v,w,x$ induce one or three hyperedges in addition to $\{u,v,w\}$;
since $\{u,v,x\}\in\ce \Leftrightarrow \{u,w,x\}\in\ce$, it follows
that $\{v,w,x\}\in\ce$.

\medskip

To prove that $(V,\ce)$ has the De Bruijn - Erd\H{o}s property, we may
assume that none of its lines equals $V$.  Now take any line $L$ and any 
vertex $v$ in $V\setminus L$.  All the lines $\ov{uv}$ with $u\ne v$
are pairwise distinct by ($\star$) and $L$ is distinct from
all of them since it does not contain $v$.
\end{proof}

\begin{theorem}\label{thm.no4}
  If, in a $3$-uniform hypergraph, no sub-hypergraph induced by four
  vertices has four hyperedges, then the hypergraph has the De Bruijn -
  Erd\H{o}s property.
\end{theorem}

\begin{proof}
  Let $(V,\ce)$ be a $3$-uniform hypergraph in which no four vertices
  induce four hyperedges; let $\ch$ denote this hypergraph and let $n$
  stand for the number of its vertices. Assuming that $n\ge 4$, we
  propose to prove by induction on $n$ that $\ch$ has at least $n$
  distinct lines. The induction basis, $n=4$, can be verified routinely.

\medskip

In the induction step, we may assume that some $\ov{pq}$ has at least
four vertices: otherwise the De Bruijn - Erd\H{o}s theorem guarantees
right away that $\ch$ has at least $n$ distinct lines. Enumerate the vertices
of $\ch$ as $p,v_2,v_3,\ldots , v_n$ with $v_2=q$. By the induction
hypothesis, at least $n-1$ of the lines $\ov{v_rv_s}$ are distinct. We
will complete the induction step by showing that at least one of the
lines $\ov{pv_2}, \ov{pv_3}, \ldots , \ov{pv_n}$ is distinct from all of
them.

\medskip

For this purpose, assume the contrary: each $\ov{pv_i}$ equals some
$\ov{v_{r(i)}v_{s(i)}}$.  Under this assumption, we are going to find
four vertices inducing four hyperedges. To begin with, we may assume
that one of $r(i),s(i)$ must be $i$: otherwise
$p,v_i,v_{r(i)},v_{s(i)}$ are four vertices inducing four hyperedges
and we are done. It follows that we may set $r(i)=i$, and so
$\ov{pv_i}=\ov{v_iv_{s(i)}}$ for all $i$. Let us write $j$ for $s(2)$.

\smallskip

{\sc Case 1:} $s(j)=2$.\\
In this case, $\ov{pv_2}=\ov{v_2v_j}=\ov{pv_j}$; by assumption, this
line has at least four vertices; $p,v_2,v_j$, and any one of its other
vertices induce four hyperedges, a contradiction.

\smallskip

{\sc Case 2:} $s(j)\ne 2$.\\
In this case, $p,v_2,v_j,v_{s(j)}$ are four vertices.  We have
$\{p,v_2,v_j\}\in\ce$ since $\ov{pv_2}=\ov{v_2v_j}$ and we have
$\{p,v_j,v_{s(j)}\}\in\ce$ since $\ov{pv_j}=\ov{v_jv_{s(j)}}$; now
$v_2\in \ov{pv_j}$, and so $\ov{pv_j}=\ov{v_jv_{s(j)}}$ implies
$\{v_2,v_j,v_{s(j)}\}\in\ce$; next, $v_{s(j)}\in \ov{v_2v_j}$, and so
$\ov{pv_2}=\ov{v_2v_j}$ implies $\{p,v_2,v_{s(j)}\}\in\ce$. But then
$p,v_2,v_j,v_{s(j)}$ induce four hyperedges, a contradiction.
\end{proof}

For all sufficiently large $n$ (certainly for all $n$ at least $27$
and possibly for all $n$), the conclusion of Theorem~\ref{thm.no4} can
be strengthened: the hypergraph has at least as many distinct
lines as it has vertices whether or not one of its lines consists of
all its vertices. In fact, the number of distinct lines grows much faster with
the number of vertices:

\begin{theorem}\label{thm.no4++}
  If, in a $3$-uniform hypergraph with $n$ vertices, no sub-hypergraph
  induced by four vertices has four hyperedges, then the hypergraph
  has at least $(n/3)^{3/2}$ distinct lines.
\end{theorem}
\begin{proof}
  Let $(V,\ce)$ denote the hypergraph and let $m$ denote the number of
  its distinct lines. We will proceed by induction on $n$. For the induction
  basis, we choose the range $n\le 3$, where the inequality $m\ge
  (n/3)^{3/2}$ holds trivially. In the induction step, consider a
  largest set $S$ of unordered pairs of distinct vertices such that
  all the lines $\ov{vw}$ with $\{v,w\}\in S$ are identical and write
  $s=\abs{S}$.

\medskip

{\sc Case 1:} $s \le n^{1/2}+1$.\\
By assumption of this case, we have
\[
m\;\ge \;\frac{\tbinom{n}{2}}{n^{1/2}+1}\;=\;
\tfrac{1}{2} n^{3/2}\cdot\frac{n^{1/2}}{n^{1/2}+1}\cdot\frac{n-1}{n}\,;
\]
since $n\ge 4$, we have
\[
\tfrac{1}{2} n^{3/2}\cdot\frac{n^{1/2}}{n^{1/2}+1}\cdot\frac{n-1}{n}
\;\ge \; \tfrac{1}{2} n^{3/2}\cdot\frac{2}{3}\cdot\frac{3}{4}
\;=\;\tfrac{1}{4} n^{3/2} \;>\; \!\left(\frac{n}{3}\right)^{3/2}\!\!\!\!.
\]
{\sc Case 2:} $s > n^{1/2}+1$.\\
By assumption of this case and since $n\ge 4$, we have $s>3$. Every
two pairs in $S$ must share a vertex (else the four vertices would
induce four hyperedges); since $s>3$, it follows that there is a
vertex common to all the pairs in $S$, and so these pairs can be
enumerated as $\{u,v_1\}, \{u,v_2\}, \ldots ,\{u,v_s\}$. We are going
to prove that
\begin{itemize}
\item[($\star$)] each of the lines $\ov{v_iv_j}$ with $1\le i<j \le s$
  is uniquely defined
\end{itemize}
in the sense that $\ov{v_iv_j}=\ov{xy}\;\Ra  \; \{x,y\}=\{v_i,v_j\}$.
To do this, consider vertices $v_i,v_j,x,y$ such that
$\ov{v_iv_j}=\ov{xy}$. These vertices cannot be all distinct (else
they would induce four hyperedges), and so symmetry lets us assume
that $x=v_i$; we will derive a contradiction from the assumption that
$y\ne v_j$.  Since $y\in\ov{v_iy}=\ov{v_iv_j}$, we have
$\{y,v_i,v_j\}\in\ce$; since $\ov{v_iy}= \ov{v_iv_j}\ne \ov{uv_i}$, we
have $y\ne u$; since $\ov{uv_i}=\ov{uv_j}$, we have
$\{u,v_i,v_j\}\in\ce$; now $u\in\ov{v_iv_j}=\ov{v_iy}$, and so
$\{u,v_i,y\}\in\ce$; finally, $y\in\ov{uv_i}=\ov{uv_j}$ implies
$\{y,u,v_j\}\in\ce$. But then the four vertices $u,v_i,v_j,y$ induce
four hyperedges; this contradiction completes our proof of ($\star$).

\medskip

Let $\cl_1$ denote the set of all lines $\ov{v_iv_j}$, let $\cl_2$
denote the set of all lines $\ov{xy}$ with $x,y\not\in \{v_1,v_2,
\ldots ,v_s\}$, and let us set $c=3^{-3/2}$. By ($\star$), we have
$\abs{\cl_1}=\binom{s}{2}$ and $\cl_1\cap \cl_2=\emptyset$; by the
induction hypothesis, we have $\abs{\cl_2}\ge c(n-s)^{3/2}$; it
follows that
\[
m\ge \tbinom{s}{2} + c(n-s)^{3/2}> \tbinom{s}{2} + cn^{3/2} - \tfrac{3}{2}cn^{1/2}s 
> cn^{3/2} +\tfrac{1}{2}n^{1/2}s(1-3c) > cn^{3/2}.
\]
\end{proof}

For large $n$, the constant $3^{-3/2}$ in the lower bound of
Theorem~\ref{thm.no4++} can be improved by more careful analysis, but
the magnitude of this lower bound, $n^{3/2}$, is the best possible. To
see this, consider the hypergraph $(V_1\cup \ldots \cup V_k,\,\ce)$, where
$V_1$, \ldots, $V_k$ are pairwise disjoint, $\{u,v,w\}\in\ce$ if and
only if $u,v\in V_i$, $w\in V_j$, $i<j$, and $u\ne v$. Here, no four
vertices induce four hyperedges; the lines are all the sets
$\{u,v\}\cup V_{i+1}\cup \ldots \cup V_k$ and all the sets $V_i\cup\{w\}$
such that $w\in V_{i+1}\cup \ldots \cup V_k$; when $\abs{V_i}=k$ for all
$i$, their total number is $k^2(k-1)$.

\medskip

In a sense, Theorem~\ref{thm.no4++} is the only theorem of its kind:
in the hypergraph $(V,\binom{V}{3})$, every sub-hypergraph induced by
four vertices has four hyperedges and the hypergraph has only one
line.

\medskip

Combining Theorems \ref{thm.no2},
\ref{thm.no1no3}, \ref{thm.no4} suggests the following questions: 
\begin{question}
  True or false? If, in a {\rm $3$}-uniform hypergraph, every
  sub-hypergraph induced by four vertices has at least two hyperedges,
  then the hypergraph has the De Bruijn - Erd\H os property.
\end{question}

\begin{question}
  True or false? If, in a {\rm $3$}-uniform hypergraph, every
  sub-hypergraph induced by four vertices has one or two or four
  hyperedges, then the hypergraph has the De Bruijn - Erd\H os
  property.
\end{question}

\bigskip

{\bf\Large  Acknowledgment}\\

\noindent The work whose results are reported here began at a workshop
held at Concordia University in June 2011.  We are grateful to the
Canada Research Chairs program for its generous support of this
workshop. We also thank Luc Devroye, Fran\c cois Genest, and Mark
Goldsmith for their participation in the workshop and for stimulating
conversations.

\end{document}